\documentclass{article}
\usepackage[english]{babel}
\usepackage[utf8]{inputenc}
\usepackage[T1]{fontenc}
\usepackage{graphicx}
\usepackage{geometry}
\usepackage{sectsty}
\usepackage{amsthm,amsfonts}
\usepackage{amsmath}
\usepackage{mathabx}
\usepackage{accents}
\usepackage{titlesec}
\usepackage{mathtools}
\mathtoolsset{showonlyrefs=true}

\titleformat*{\section}{\Large\bfseries}
\titleformat*{\subsection}{\large\bfseries}
\titleformat*{\subsubsection}{\large\bfseries}
\titleformat*{\paragraph}{\large\bfseries}
\titleformat*{\subparagraph}{\large\bfseries}

\newtheorem{teo}{Theorem}
\newtheorem{lema}[teo]{Lemma}
\newtheorem{cor}[teo]{Corollary}
\newtheorem{prop}[teo]{Proposition}

\newtheoremstyle{mytheoremstyle} 
{\topsep}                    
{\topsep}                    
{}                   
{}                           
{\scshape}                   
{.}                          
{.5em}                       
{}  

\theoremstyle{mytheoremstyle} \newtheorem{nota}{Remark}
\theoremstyle{mytheoremstyle} 
\numberwithin{equation}{section}
\newcommand{\real}{\mathbb{R}}

\newcommand \ben {\begin{equation}}
\newcommand \een {\end{equation}}
\newcommand \be {\begin{equation*}}
\newcommand \ee {\end{equation*}}
\newcommand \bi {\begin{itemize}}
\newcommand \ei {\end{itemize}}

\DeclareMathOperator*{\parteim}{Im}

\DeclareMathOperator*{\dist}{dist}

\DeclareMathOperator*{\supp}{supp}
\DeclareMathOperator*{\NLS}{NLS}

\title{\textbf{Finite speed of disturbance for the nonlinear Schrödinger equation}}
\author{Simão Correia}
\begin{document}
\maketitle
\begin{abstract}
We consider the Cauchy problem for the nonlinear Schrödinger equation on the whole space. After introducing a weaker concept of finite speed of propagation, we show that the concatenation of initial data gives rise to solutions whose time of existence increases as one translates one of the initial data. Moreover, we show that, given global decaying solutions with initial data $u_0, v_0$, if $|y|$ is large, then the concatenated initial data $u_0+v_0(\cdot -y)$ gives rise to globally decaying solutions.
\end{abstract}
\section{Introduction}
In this work, we consider the classical nonlinear Schrödinger equation in $\real^d$,
\begin{equation}\tag{NLS}\label{NLS}
iu_t + \Delta u + \lambda|u|^\sigma u=0,
\end{equation}
where $\lambda\in\real$ and $0<\sigma<4/(d-2)^+$. The initial value problem for $u_0\in H^1(\real^d)$ is locally well-posed (see, for example, the monograph \cite{cazenave}; also \cite{kato} and \cite{kato1}). If the corresponding solution is defined for all positive times, one says it is \textit{global}. Otherwise, the maximal time of existence $T(u_0)$ is finite and the solution \textit{blows up} at $t=T(u_0)$:
$$
\lim_{t\to T(u_0)} \|\nabla u(t)\|_2=+\infty.
$$
The initial value problem for $u_0\in H^s(\real^N)$, $0\le s <1$ can also be considered. Under the hypothesis of $H^s$-\textit{subcriticallity}, that is,
$
\sigma <4/(d-2s)^+,
$
the analogous local well-posedness result holds, with the blow-up alternative
$$
\lim_{t\to T(u_0)} \|u(t)\|_{H^s}=+\infty.
$$
In the critical case $\sigma=4/(N-2s)$, one has the so-called \textit{conditional local well-posedness}, where the blow-up alternative is replaced by
$$
\|u\|_{L^\infty((0,T(u_0)), H^s)} + \|u\|_{L^\gamma((0,T(u_0)), B^{s}_{\rho,2})} = +\infty,
$$
for some specific $\gamma$ and $\rho$. On the other hand, the proof of local well-posedness for critical cases usually yields a proof of global existence for small data. 

Notice that we purposely omitted the dependence of the time of existence on $s$. In fact, using standard persistence arguments, one may prove that the time of existence in $H^{s'}$ for $u_0\in H^s$ is the same for all $0\le s' \le s$. As a consequence, $H^s$-subcritical cases will also have global existence for small data.

There are various situations where one may bound \textit{a priori} the solution, thus ensuring that it is global. Examples include the global existence in $H^1$ for the defocusing case $\lambda<0$ or under the $L^2$-subcritical case. On the other hand, in the focusing case $\lambda>0$ and in the $L^2$-(super)critical case, one may prove indirectly the existence of blow-up through a Virial argument (cf. \cite{glassey}): if $xu_0\in L^2$,
$$
\frac{d^2}{dt^2}\|xu(t)\|_2^2\le 16E(u_0) :=16\left(\frac{1}{2}\|\nabla u_0\|_2^2 - \frac{\lambda}{\sigma+2}\|u_0\|_{\sigma+2}^{\sigma+2}\right),\quad t<T(u_0).
$$
Hence, for $E(u_0)<0$, if $T(u_0)=+\infty$, the quantity $\|xu(t)\|_2$ would become negative at some point, which is absurd.
\vskip15pt
The Schrödinger equation has \textit{infinite speed of propagation}: the information at one point $x\in\real^d$ can influence points at arbitrary distance. This may be seen in the linear equation either by taking as initial condition the Dirac delta or by observing that larger frequencies travel faster. In fact, it has been proven in \cite{kenigponcevega} that the only solution of (NLS) with compact support at two different times is the zero solution. Finite speed of propagation is a useful tool to obtain qualitative results on the dynamics of an equation, the classical example being the wave equation. One of these properties is \textit{localization}, \textit{i.e.}, to study what happens near a point $x\in\real^d$, one only needs to look at the backward light cone. 

A similar application of finite speed of propagation is the \textit{concatenation} of initial data: if one takes two compactly supported initial data $u_0, v_0$, then, given $v\in\real^d$, the solution with initial condition $u_0+v_0(\cdot - y)$, $|y|$ large, will behave like the sum of the individual solutions up to a large time $T_y$. Moreover, if the nonlinear effects have already dissipated by time $T_y$, one expects that the solution is global. Results of this type are currently unavailable for the \eqref{NLS}.

The aim of this work is to give a weaker notion of speed of propagation and use it to obtain concatenation results for the (NLS). Finite speed of disturbance is an estimate for the amount of information that appears in some observation set, taking into account the distance between that set and the initial support. The notion will be introduced in Section \ref{sec:finitespeed}. 
\vskip15pt
We now state the main result of this paper. Since global existence in the $L^2$-subcritical case is already known for any initial data, we focus on the $L^2$-(super)critical case $\sigma\ge 4/d$. In what follows, $\NLS(u_0)$ will denote the maximal solution of (NLS) with initial condition $u_0$. Given any time interval $I$, we set
$$
\|u\|_{S^s(I)} = \|u\|_{L^\infty(I,H^s)} + \|u\|_{L^{\gamma}(I, B^{s}_{\rho,2})}, \quad \rho=\frac{d(\sigma+2)}{d+s\sigma},\ \gamma=\frac{4(\sigma+2)}{\sigma(d-2s)}, \ 0<s<1,
$$
and, for $q=4(\sigma+2)/\sigma d$,
$$
\|u\|_{S^0(I)}= \|u\|_{L^\infty(I, L^2)} + \|u\|_{L^{\sigma+2}(I, L^{\sigma+2})},\quad \|u\|_{S^1(I)}= \|u\|_{L^\infty(I, H^1)} + \|u\|_{L^{q}(I, W^{1, \sigma+2})}
$$

Define the set of global decaying solutions with bounded Strichartz norms (up to order one) as
$$
\mathcal{GD}=\left\{ u_0\in H^1(\real^d): T(\NLS(u_0))=\infty,\ \|\NLS(u_0)\|_{S^1(0,\infty)}<\infty \right\}.
$$
As proven in \cite[Theorem 6.2.1]{cazenave}, there exists $\delta>0$ such that
$$
\{u_0\in H^1(\real^d): \|u_0\|_{H^1}<\delta\} \subset \mathcal{GD}.
$$
\begin{teo}[Concatenation of initial data]\label{teocritico}
Set $\sigma = 4/d$ or $\sigma\ge\min\{1,4/d\}$. Given initial data $u_0, v_0\in H^1$,  a fixed time $T<T(u_0), T(v_0)$ and $\epsilon>0$, there exists $D_T>0$ such that, for any $w_0\in H^1$ small enough, 
$$
T(u_0+v_0(\cdot-y) + w_0)>T, \quad |y|>D_T
$$
and, taking $s$ such that $\sigma=4/(d-2s)$,
$$
\| \NLS(u_0+v_0(\cdot-y) + w_0) - \NLS(u_0)-\NLS(v_0(\cdot-y))\|_{S^s(0,T)} <\epsilon.
$$
Moreover, if $u_0, v_0\in \mathcal{GD}$,
there exists $D_\infty>0$ such that $u_0+v_0(\cdot - y) + w_0\in \mathcal{GD}$, $|y|>D_\infty$, and 
$$
\| \NLS(u_0+v_0(\cdot-y) + w_0) - \NLS(u_0)-\NLS(v_0(\cdot-y))\|_{S^s(0,\infty)} <\epsilon.
$$
\end{teo}

\begin{nota}
The value of $D_\infty$ depends on the global bound $M$ for $\NLS(u_0)$ and $\NLS(v_0)$, on the size of the tails of $u_0$ and $v_0$, and also on the time $T$ for which
$$
\|\NLS(u_0)\|_{S^s(T,\infty)}, \|\NLS(v_0)\|_{S^s(T,\infty)}\quad \mbox{ is small enough}.
$$
\end{nota}

\begin{nota}
Independently of the spatial dimension, the result is always true for $\sigma=4/d$. In the supercritical case, one could try to prove the concatenation result with similar arguments to those of \cite[Theorem 6.2.1]{cazenave}. However, the information given by the finite speed of disturbance concerns the solution itself and not its derivatives. As a consequence, one must try to prove global well-posedness by using as little derivatives as possible. This is achieved over the critical space $H^s$, with $\sigma=4/(d-2s)$ and $s>0$. To estimate properly the interaction between the two solutions in Besov spaces, we require that $\sigma\ge 1$.
\end{nota}

\begin{nota}
Observe that the concatenation of two initial data with positive energy, for large translations, also has positive energy. Thus the concatenation result does not contradict the Virial blow-up argument.
\end{nota}

It is important to notice that the second part of Theorem \ref{teocritico} can be iterated: one starts with two global solutions with linear decay and builds a new global solution with linear decay. Moreover, the $L^2$ and $H^1$ norms of the new initial data is the sum of the corresponding norms of the given data. As a consequence, one obtains
\begin{cor}\label{cor:largeglobal}
Set $\sigma = 4/d$ or $\sigma\ge \min\{1,4/d\}$. Given $K_1, K_2>0$, there exists $u_0\in \mathcal{GD}$ compactly supported such that
$$
\|u_0\|_2=K_1,\quad \|\nabla u_0\|_2=K_2.
$$
\end{cor}

Our result indicates that blow-up behaviour is necessarily connected with how much localized is the initial data: if the initial data is made up of small $H^1$ pieces, sufficiently spread out in space, the corresponding solution is global. On the other hand, we point out that the Virial blow-up argument is also connected with the localization of the initial data, through both variance and energy. Though far from concrete, an underlying necessary and sufficient condition for blow-up becomes apparent.

\begin{nota}
The question of concatenation of global solutions with no decay properties is not a trivial matter: assume that Theorem \ref{teocritico} is applicable to global solutions without decay properties. If one could choose $v_0=0$ and $u_0=Q$, where $Q$ is the ground-state of (NLS), then the perturbed concatenation $u_0+v_0(\cdot-y)+w_0 =Q+w_0$ would give rise to solutions whose time of existence goes to infinity as $|y|\to \infty$. Therefore the solution with initial data $Q+w_0$ would be global for any $w_0\in H^1$ small enough. However, this contradicts the known instability result of ground-states of \cite{berescaz}.
\end{nota}

\vskip15pt
\textbf{Notation.} $L^p$ norms over all of $\real^d$ is be denoted by $\|\cdot\|_p$. Moreover, if the domain of integration is $\real^d$, we will often ommit it. We define 
$B_R(0)=\{x\in\real^d: |x|<R\}$. Finally, $\dist(A,B)$ is the distance between the sets $A$ and $B$.

%

\section{Finite speed of disturbance}\label{sec:finitespeed}

Let us start with the linear equation
\begin{equation}\tag{LS}\label{LS}
iu_t + \Delta u=0,\quad u(0)=u_0\in H^1(\real^d).
\end{equation}
If one takes $\phi\in W^{1.\infty}(\real^d)$ real-valued, then
\begin{equation}
\frac{1}{2}\frac{d}{dt}\int \phi^2|u|^2 = 2\parteim \int \phi\overline{u}\nabla \phi\cdot\nabla u\le 2\|\phi u\|_2\|\nabla \phi\|_\infty\|\nabla u\|_2.
\end{equation}
Integrating this differential inequality,
\begin{equation}
\|\phi u(t)\|_2 \le \|\phi u_0\|_2 + 2t\|\nabla \phi\|_\infty \sup_{s\in [0,t]}\|\nabla u(s)\|_2 = \|\phi u_0\|_2 + 2t\|\nabla \phi\|_\infty \|\nabla u_0\|_2,
\end{equation}
since the $L^2$ norm of the gradient is preserved by \eqref{LS}. Now take two disjoint smooth open sets $A, B\subset \real^d$ and take $\phi\in W^{1,\infty}$ such that
$$
\phi \equiv 1\mbox{ on }B,\quad \phi \equiv 0\mbox{ on }A,\quad \|\nabla \phi\|_\infty <\frac{1}{\dist(A,B)}.
$$
Then
$$
\|u(t)\|_{L^2(B)}\le \frac{2t\|\nabla u_0\|_2}{\dist(A,B)} + \|u_0\|_{L^2(\real^d\setminus A)},\quad A,B\subset \real^d.
$$
In the special case where $\phi u_0\equiv 0$, one has 
\begin{equation}\label{eq:desigualphi}
\|\phi u(t)\|_2 \le 2\|\nabla \phi\|_\infty \|\nabla u_0\|_2 t.
\end{equation}
and so one obtains the
\textit{finite speed of disturbance} for the \eqref{LS}:
\begin{equation}\label{eq:finitespeeddisturbancelinear}
\|u(t)\|_{L^2(B)}\le \frac{2\|\nabla u_0\|_2}{\dist(A,B)}t.
\end{equation}
This inequality tells us that, even though information may travel at any speed, the \textit{amount of information} that reaches some set $B$ grows (at most) linearly in time, with growth factor inversely proportional to the distance between the source of the information and the observation set. In another way, even though the higher frequencies travel faster, they carry a controlled amount of mass.
\begin{nota}
Fix $t\ge 0$. Given any $\gamma(t)\ge 1$, the choice $B=\real^d\setminus(A+B_{\gamma(t)}(0))$ in \eqref{eq:finitespeeddisturbancelinear} yields
$$
\|u(t)\|_{L^2(\real^d\setminus(A+B_{\gamma(t)}(0)))}\le \frac{2\|\nabla u_0\|_2t}{\gamma(t)}.
$$
For any $\epsilon>0$, if $\gamma(t)=\gamma t$, $\gamma=2\|\nabla u_0\|_2 /\epsilon $, we see that
$$
\|u(t)\|_{L^2(\real^d\setminus(A+B_{\gamma t}(0)))}\le \epsilon.
$$
This means that most of the total mass lies inside a specific cone of light, with speed given by the initial kinetic energy.
\end{nota}

\begin{nota}
Suppose that $B$ is such that, for some unit vector $v\in\real^d$,
$$\dist(A,B+vt)=\dist(A,B)+t,\quad  t>0.$$
Using the Galilean invariance $$u_b(x,t)=e^{i\frac{bv}{2}\left(x-\frac{bv}{2}t\right)}u(t,x-bvt),\quad b>0,
$$
one has, for any set $C$ with $\dist(A,C)>0$,
$$
\|u_b(t)\|_{L^2(C)}\le \frac{2\|\nabla u_b(0)\|_2}{\dist(A,C)}t.
$$
For each fixed $t>0$, take $C=B+bvt$. One then arrives at a more general estimate for the speed of disturbance
$$
\|u(t)\|_{L^2(B)}=\|u_v(t)\|_{L^2(B+bvt)}\le \frac{t}{\dist(A,B)+bt}\left(4\|\nabla u_0\|_2^2+b^2\|u_0\|_2^2\right)^{1/2},\quad b>0.
$$
Notice that, taking both $b, t\to\infty$, one has the trivial bound $\|u(t)\|_{L^2(B)}\le \|u_0\|_2$.
\end{nota}

In the (NLS) case, analogous computations yield 
\begin{equation}\label{eq:desigualphi2}
\|\phi u(t)\|_2 \le 2\|\nabla \phi\|_\infty \left(\sup_{s\in[0,t]}\|\nabla u(s)\|_2\right) t + \|\phi u_0\|_2.
\end{equation}
Choosing $\phi$ as in the linear case, one has
$$
\|u(t)\|_{L^2(B)}\le \frac{2t\sup_{s\in [0,t]} \|\nabla u(s)\|_2}{\dist(A,B)} + \|u_0\|_{L^2(\real^d\setminus A)},\quad A,B\subset \real^d.
$$
in the general case and, if $\supp u_0\subset A$, one obtains the finite speed of disturbance for the (NLS)
\begin{equation}\label{eq:finitespeeddisturbancenaolinear}
\|u(t)\|_{L^2(B)}\le \frac{2\sup_{s\in[0,t]}\|\nabla u(s)\|_2}{\dist(A,B)}t.
\end{equation}
In this case, we see that, as long as $u$ remains bounded in $H^1$, the solution will have finite speed of disturbance and the previous considerations are still valid. Another useful estimate can be obtained from \eqref{eq:desigualphi2}: using Gagliardo-Nirenberg's inequality,
\begin{align}
\|u(t)\|_{L^{\sigma+2}(B)}&\le\|\phi u(t)\|_{{\sigma+2}} \lesssim \|\phi u(t)\|_2^{1-\frac{ d\sigma}{2(\sigma+2)}}\|\nabla (\phi u(t))\|_2^\frac{d\sigma}{2(\sigma+2)}\\\label{eq:finitespeedLp}&\lesssim \left(\frac{2t\sup_{s\in [0,t]} \|\nabla u(s)\|_2}{\dist(A,B)} + \|u_0\|_{L^2(\real^d\setminus A)}\right)^{1-\frac{ d\sigma}{2(\sigma+2)}}\|\nabla (\phi u(t))\|_2^\frac{d\sigma}{2(\sigma+2)}\\&\lesssim \left(\frac{2t\sup_{s\in [0,t]} \|\nabla u(s)\|_2}{\dist(A,B)} + \|u_0\|_{L^2(\real^d\setminus A)}\right)^{1-\frac{ d\sigma}{2(\sigma+2)}}\\&\hskip15pt\times\left(\|\nabla u(t)\|_2^2 + \frac{1}{\dist(A,B)^2}\|u_0\|_2^2\right)^\frac{d\sigma}{4(\sigma+2)}.
\end{align}

\begin{nota}
An estimate similar to \eqref{eq:desigualphi2} has been used in \cite{merlemartel} to understand the interference between solitons centered at distant points and build multi-soliton solutions in the $L^2$-subcritical case. Another instance of such an estimate has also been used in \cite{ginibreveloscat} to show asymptotic completeness in $H^1$ of the defocusing (NLS).
\end{nota}

\begin{nota}
Consider the defocusing case $\lambda<0$ in the $L^2$-critical case. As it is well-known, given any initial data $u_0\in H^1(\real^d)\cap L^2(|x|^2 dx)=: \Sigma$, the corresponding solution $u=\NLS (u_0)$ scatters to a linear solution, \textit{i.e.}, there exists a unique $u_+\in \Sigma$ such that
$$
\|S(-t)u(t) - u_+\|_\Sigma \to 0,\quad t\to\infty.
$$
Thus one may define the forward scattering operator as the mapping $u_0\mapsto u_+$. We observe that the simple application of finite speed of disturbance can lead to an estimate for the scattering operator. In fact, using the \textit{lens transform} (see \cite{carles}), if $v$ is the solution of the nonlinear Schrödinger equation with an harmonic potential
$$
iv_t + \Delta v - |x|^2v + \lambda |v|^{4/d}v=0, \quad v(0)=u_0,
$$
then the Fourier transform $(u_+)^\wedge$ of $u_+$ is precisely $v(\pi/2)$ (this has been observed in \cite{tao}). This equation also enjoys finite speed of disturbance:
$$
\|v(t)\|_{L^2(B)}\le \frac{2t\sup_{s\in [0,t]} \|\nabla v(s)\|_2}{\dist(A,B)} + \|u_0\|_{L^2(\real^d\setminus A)},\quad A,B\subset \real^d.
$$
Moreover, by conservation of energy, one has
$$
\|\nabla v(t)\|_2 \lesssim \|v_0\|_\Sigma = \|u_0\|_\Sigma, \ t>0.
$$
Thus, taking $t=\pi/2$, one obtains
$$
\|(u_+)^\wedge\|_{L^2(B)} \le \frac{\pi \|u_0\|_\Sigma}{\dist(A,B)} + \|u_0\|_{L^2(\real^d\setminus A)},\quad A,B\subset \real^d.
$$
In particular, the localization of the initial data on the physical side implies a localization of the scattering state $u_+$ on the frequency side.
\end{nota}

\begin{nota}
Finite speed of perturbation, being a weaker version of the classical finite speed of propagation, can be observed in a larger number of equations. In particular, it would be interesting to study this concept for other dispersive PDE's for which one has infinite speed of propagation, such as the Korteweg-de-Vries equation.
\end{nota}

\section{Proof of the main result}

\begin{proof}[Proof of Theorem \ref{teocritico} for $\sigma=4/d$]
\textit{Step 1.} Before we proceed, we make a number of simplifications using the symmetries of the (NLS). First, we may consider $y=2De_1$, where $e_1\in\real^d$ is the first element of the canonic basis of $\real^d$. Moreover, due to the translation invariance, we may prove the result for the initial data $u_0(\cdot + y/2) + v_0(\cdot-y/2)=u_0(\cdot + De_1) + v_0(\cdot-De_1)$. In what follows, $\delta(D)$ will denote a decreasing function such that
$$
\delta(D)\to 0, \quad D\to\infty.
$$
Define $A^-=\{x\in \real^d: x_1<-D/2\}$ and $A^+=\{x\in\real^d: x_1>D/2 \}$. Since $u_0, v_0\in L^2$, 
\begin{equation}\label{eq:dadoscaudapequena}
\|u_0(\cdot+De_1)\|_{L^2(\real^d\setminus A^-)}, \ \|v_0(\cdot-De_1)\|_{L^2(\real^d\setminus A^+)}<\delta(D).
\end{equation}
Set $u=\NLS(u_0(\cdot+De_1))$, $v=\NLS(v_0(\cdot-De_1))$ and consider the initial value problem
$$
iw_t + \Delta w + |u+v+w|^{\sigma}(u+v+w)-|u|^\sigma u - |v|^\sigma v=0, w(0)=w_0\in H^1.
$$
Notice that
$$
\NLS(u_0(\cdot + De_1)+v_0(\cdot -De_1) + w_0)= u+v+w
$$
for as long as any three solutions exist. As a consequence, the local existence of $w$ as a $L^2$-solution is a trivial matter.

\vskip10pt
\noindent\textit{Step 2.} Since $T<T(u_0), T(v_0)$, there exists $M>0$ such that
\begin{equation}\label{eq:uniformboundT}
\|u\|_{L^\infty((0,T), H^1)}, \|v\|_{L^\infty((0,T), H^1)}\le M.
\end{equation}
One has
$$
||u+v+w|^{\sigma}(u+v+w)-|u|^\sigma u - |v|^\sigma v|\lesssim |v|^\sigma|u| + |u|^\sigma|v| + |u|^\sigma |w| + |v|^\sigma |w| + |w|^{\sigma+1}
$$
We write $B^+=\{x\in\real^d: x_1> 0\}$ and $B^-=\{x\in\real^d: x_1<0\}$. The idea is that, due to \eqref{eq:dadoscaudapequena}, $u$ is small over $B^+$ and $v$ is small over $B^-$.

Setting $\rho=(\sigma+2)/(\sigma+1)$,
\begin{align}\label{eq:estimativafonte}
\||u|^\sigma v\|_{L^\rho}&\lesssim \||u|^\sigma v\|_{L^{\rho}(B^+)}+\||u|^\sigma v\|_{L^\rho(B^-)} \\&\lesssim \|u\|_{L^{\sigma+2}(B^+)}^\sigma \|v\|_{L^{\sigma+2}(B^+)} + \|u\|_{L^{\sigma+2}(B^-)}^\sigma \|v\|_{L^{\sigma+2}(B^-)}=I_1 +I_2
\end{align}
We estimate $I_1$: we use \eqref{eq:finitespeedLp} on the first term and the uniform bound \eqref{eq:uniformboundT},
\begin{align*}
I_1\lesssim \|u\|_{L^{\sigma+2}(B^+)}^{\sigma}\|v\|_{L^{\sigma+2}} &\lesssim \left(\frac{TM}{\dist(A^-, B^+)} + \|u_0\|_{L^2(\real^d\setminus A^-)}\right)^{\frac{\sigma^2}{\sigma+2}}M^{\frac{2\sigma}{\sigma+2}+1}\\ &\lesssim \left(\frac{TM}{D} + \delta(D)\right)^{\frac{\sigma^2}{\sigma+2}}M^{\frac{2\sigma}{\sigma+2}+1}\\&\lesssim \delta(D).
\end{align*}
The same reasoning can be applied to $I_2$, which implies that $\||u|^\sigma v\|_{L^\rho}<\delta(D)$. Analogously, we also have $\||v|^\sigma u\|_{L^{\rho}} \lesssim \delta(D)$.

We apply Strichartz estimates to the Duhamel formula for $w$ on a fixed interval $(0,t)$:
\begin{align*}
\|w\|_{S^0(0,t)}&\lesssim \|w_0\|_2 + t^{\frac{1}{\rho}}\delta(D) + \||u|^\sigma |w|\|_{L^\rho((0,t),L^{\rho})} \\&+\||v|^\sigma |w|\|_{L^\rho((0,t),L^{\rho})}+ \|w\|_{S^0(0,t)}^{\sigma+1}\\&\lesssim  \|w_0\|_2+t^{\frac{1}{\rho}}\delta(D)  + \|u\|_{L^{\sigma+2}((0,t),L^{\sigma+2})}^\sigma\|w\|_{S^0(0,t)} \\&+ \|v\|_{L^{\sigma+2}((0,t),L^{\sigma+2})}^\sigma\|w\|_{S^0(0,t)}+ \|w\|_{S^0(0,t)}^{\sigma+1}\\&\lesssim  \|w_0\|_2+t^{\frac{1}{\rho}}\delta(D)  +t^\frac{\sigma}{\sigma+2}M^\sigma \|w\|_{S^0(0,T)} + \|w\|_{S^0(0,t)}^{\sigma+1}.
\end{align*}
We choose $T_0$ such that
$$
T_0^\frac{\sigma}{\sigma+2}M^\sigma\lesssim \frac{1}{2},
$$
so that
$$
\|w\|_{S^0(0,t)}\lesssim  \|w_0\|_2+\delta(D) + \|w\|_{S^0(0,t)}^{\sigma+1},\quad t<T_0.
$$
A standard obstruction argument then implies that there exists $\eta>0$ small such that, if 
$$ \|w_0\|_2+\delta(D)<\eta',\ \eta'<\eta, 
$$ 
then $w$ exists (as an $L^2$ solution) up to time $T_0$ and $\|w(T_0)\|_2<2\eta'$. This process may be iterated as long as the $L^2$ norm of $w$ remains below $\eta$. Thus, for sufficiently small $\|w_0\|_2$ and large $D$, one guarantees that $w$ exists up to time $T$ and that 
$$
\|w\|_{S^0(0,T)}<\eta<\epsilon.
$$
This concludes the proof of the first part of Theorem \ref{teocritico}. 
\vskip10pt
\noindent\textit{Step 3.} For the second part of the Theorem, fix $\delta>0$ small and choose $T$ large enough such that
$$
\|u\|_{S^0(T,\infty)}, \|v\|_{S^0(T,\infty)}\le \delta.
$$
Applying the first part of the Theorem, for $D_T$ large, $w$ is defined up to time $T$ with
$$
\|w\|_{S^0(0,T)}\le \delta.
$$
Recalling that $\rho=(\sigma+2)/(\sigma+1)$, one easily checks that
$$
\||u|^\sigma v\|_{L^{\rho}((T,\infty), L^{\rho})} \lesssim \|u\|_{S^0(T,\infty)}^\sigma \|v\|_{S^0(T,\infty)}\le \delta, 
$$
$$
\||v|^\sigma u\|_{L^{\rho}((T,\infty), L^{\rho})} \lesssim \|v\|_{S^0(T,\infty)}^\sigma \|u\|_{S^0(T,\infty)}\le \delta.
$$
Let $T^*$ be the maximal time of existence of $w$. Applying Strichartz estimates to the Duhamel formula
$$
w(t) = S(t-T)w(T) -i\lambda \int_T^t S(t-s)\left(|u+v+w|^{\sigma}(u+v+w)-|u|^\sigma u - |v|^\sigma v\right)ds,
$$
with $T<t<T^*$, one has
\begin{align*}
\|w\|_{S^0(T,t)}&\lesssim \|w(T)\|_2 + \delta + \||u|^\sigma |w|\|_{L^\rho((T,t),L^{\rho})} \\&+\||v|^\sigma |w|\|_{L^\rho((T,t),L^{\rho})}+ \|w\|_{S^0(T,t)}^{\sigma+1}\\&\lesssim  \delta + \|u\|_{L^{\sigma+2}((T,t),L^{\sigma+2})}^\sigma\|w\|_{S^0(T,t)} \\&+ \|v\|_{L^{\sigma+2}((T,t),L^{\sigma+2})}^\sigma\|w\|_{S^0(T,t)}+ \|w\|_{S^0(T,t)}^{\sigma+1}\\&\lesssim  \delta + \delta \|w\|_{S^0(T,t)} + \|w\|_{S^0(T,t)}^{\sigma+1},\quad T<t<T^*.
\end{align*}
Thus
$$
\|w\|_{S^0(T,t)}\lesssim \delta + \|w\|_{S^0(T,t)}^{\sigma+1},\quad T<t<T^*,
$$
which, for $\delta$ sufficiently small, implies that
$$
\|w\|_{S^0(T,t)}<\epsilon,\ T<t<T^*.
$$
The blow-up alternative (on $L^2$) now implies that $T^*=\infty$ and that $\|w\|_{S^0(0,\infty)}<\epsilon$. Moreover, since the initial data is in $H^1$, by persistence of regularity, $w$ is globally defined in $H^1$.
\end{proof}
\vskip15pt

The remainder of this section will focus on the proof of Theorem \ref{teocritico} for the supercritical case. To that end, we shall work on $H^s$, with $\sigma=4/(d-2s)$. An admissible pair of particular importance is
$$
\rho=d(\sigma+2)/(d+s\sigma), \quad \gamma=\frac{4(\sigma+2)}{\sigma(d-2s)}=\sigma+2.
$$
Recall that the Besov space $B^s_{\rho,2}$ may be endowed with the norm
$$
\|u\|_{B^s_{\rho,2}}^2 =  \|u\|_\rho^2+\|u\|_{\dot{B}^s_{\rho,2}}^2:=  \|u\|_\rho^2+\int_0^\infty \left(\tau^{-s}\sup_{|y|<\tau}\|u(\cdot-y)-u\|_\rho\right)^2\frac{d\tau}{\tau}.
$$

\begin{lema}\label{lema:injeccao}
One has
$$
W^{1,\rho}(\real^d)\hookrightarrow B^s_{\rho, 2}(\real^d)\hookrightarrow L^{\frac{\sigma\rho\rho'}{\rho-\rho'}}(\real^d).
$$
\end{lema}
\begin{proof}
The second injection is a direct consequence of Sobolev's injection. For the first, we take $u\in C^\infty_0(\real^d)$ and write
$$
\|u\|_{B^s_{\rho,2}}^2\lesssim  \|u\|_\rho^2+\int_0^1 \left(\tau^{-s}\sup_{|y|<\tau}\|u(\cdot-y)-u\|_\rho\right)^2\frac{d\tau}{\tau} + \int_1^\infty \left(\tau^{-s}\sup_{|y|<\tau}\|u(\cdot-y)-u\|_\rho\right)^2\frac{d\tau}{\tau}
$$
The characterization of Sobolev spaces using translation operators (see, for example, \cite[Proposition 8.5]{brezis}) implies that
$$
\|u(\cdot-y)-u\|_\rho\lesssim \|\nabla u\|_\rho|y|
$$
Hence
\begin{align*}
\|u\|_{B^s_{\rho,2}}^2 &\lesssim \|u\|_{\rho}^2 + \int_0^1 \left(\tau^{-s}\sup_{|y|<\tau}\|\nabla u\|_\rho|y|\right)^2\frac{d\tau}{\tau} + \int_1^\infty \left(\tau^{-s}\sup_{|y|<t}\|u\|_\rho\right)^2\frac{d\tau}{\tau}\\&\lesssim \|u\|_{W^{1,\rho}}^2\left(1+\int_0^1 \frac{d\tau}{\tau^{1-2s}} + \int_1^\infty \frac{d\tau}{\tau^{1+s}}\right)\lesssim \|u\|_{W^{1,\rho}}^2.
\end{align*}
\end{proof}

As in the first step of the proof of Theorem 1 in the critical case, from now on, $\delta(D)$ will denoted a decreasing function of $D$ such that $\delta(D)\to 0$ as $D\to\infty$.

\begin{lema}\label{lema:estimativafonte}
Fix $\sigma>4/d$ such that $\sigma\ge 1$. Given $u_0, v_0\in H^1$ and $T<T(u_0), T(v_0)$, if one writes $u=\NLS(u_0(\cdot+De_1))$ and $v=\NLS(v_0(\cdot-De_1))$, then
$$
\left\||u+v|^\sigma(u+v) - |u|^\sigma u - |v|^\sigma v\right\|_{L^{\gamma'}((0,T), B^s_{\rho',2})}\le \delta(D).
$$
\end{lema}
\begin{proof}
Set 
\begin{equation}\label{eq:termofonte}
N(u,v)=|u+v|^\sigma(u+v) - |u|^\sigma u - |v|^\sigma v
\end{equation}
and 
$$
M=\|u\|_{L^\infty((0,T), H^1)} + \|u\|_{L^\gamma((0,T), W^{1,\rho})} + \|v\|_{L^\infty((0,T), H^1)} + \|v\|_{L^\gamma((0,T), W^{1,\rho})},
$$
$$
B(t)=\|u(t)\|_{W^{1,\rho}} + \|v(t)\|_{W^{1,\rho}}.
$$
It is easy to check, using Hölder's inequality, that
$$
\|B(t)^{\sigma+1}\|_{L^{\gamma'}(0,T)} \lesssim T^{\frac{4-\sigma(d-2s)}{4}} M^{\sigma+1}
$$
Once again, write 
$$
A^-=\{x\in \real^d: x_1<-D/2\},\quad A^+=\{x\in\real^d: x_1>D/2 \},
$$
$$
B^-=\{x\in \real^d: x_1<0\},\quad B^+=\{x\in \real^d: x_1>0\}.
$$ 
Since $u_0, v_0\in L^2$, 
\begin{equation}\label{eq:dadoscaudapequena2}
\|u_0(\cdot+De_1)\|_{L^2(\real^d\setminus A^-)}, \ \|v_0(\cdot-De_1)\|_{L^2(\real^d\setminus A^+)}<\delta(D).
\end{equation}
We set $0<a,b<1$ such that
$$
\|z\|_{L^\alpha}\lesssim \|z\|_{L^\rho}^a\|\nabla z\|_{L^\rho}^{1-a}, \quad \|z\|_{L^\rho}\lesssim \|z\|_{L^2}^b\|\nabla z\|_{L^2}^{1-b},\quad \alpha=\sigma\rho\rho'/(\rho-\rho'), \ z\in C^\infty_0(\real^d).
$$

Then
\begin{align*}
\|N(u,v)\|_{\dot{B}^s_{\rho',2}}^2&=\int_0^\infty \left(\tau^{-s}\sup_{|y|<\tau}\|N(u,v)(\cdot-y)-N(u,v)\|_{\rho'}\right)^2\frac{d\tau}{\tau}\\&\lesssim \int_0^1 \left(\tau^{-s}\sup_{|y|<\tau}\|N(u,v)(\cdot-y)-N(u,v)\|_{L^{\rho'}(B^+)}\right)^2\frac{d\tau}{\tau}\\&+\int_0^1 \left(\tau^{-s}\sup_{|y|<\tau}\|N(u,v)(\cdot-y)-N(u,v)\|_{L^{\rho'}(B^-)}\right)^2\frac{d\tau}{\tau}
\\&+\|N(u,v)\|_{L^{\rho'}(B^+)}^2 +\|N(u,v)\|_{L^{\rho'}(B^-)}^2\lesssim I_1^2 + I_2^2 + I_3^2+I_4^2.
\end{align*}
We treat $I_1$ as follows: setting $C^\pm=B^\pm + B_1(0)$ and recalling that $W^{1,\rho}\hookrightarrow B^s_{\rho,2}\hookrightarrow L^\alpha$, $\alpha=\sigma\rho\rho'/(\rho-\rho')$,
\begin{align*}
\|N(u,v)(\cdot-y)&-N(u,v)\|_{L^{\rho'}(B^+)}\lesssim \left(\|u\|_{L^{\alpha}(C^+)}^{\sigma-1}+ \|v\|_{L^{\alpha}(C^+)}^{\sigma-1} \right)\\\times&\left(\|u\|_{L^{\alpha}(C^+)}\|v(\cdot-y)-v\|_{L^\rho(B^+)} + \|v\|_{L^{\alpha}(C^+)}\|u(\cdot-y)-u\|_{L^\rho(B^+)}\right)\\&\lesssim B(t)^{\sigma-1}\|u\|_{L^\alpha(C^+)}\|v(\cdot-y)-v\|_{L^\rho} + B(t)^\sigma\|u(\cdot-y)-u\|_{L^\rho(B^+)}
\end{align*}
For the first term, take a smooth cut-off function $\phi$ with $\phi\equiv 1$ over $C^+$ and $\phi\equiv 0$ over $A^-$. Then, from Gagliardo-Nirenberg and finite speed of disturbance,
\begin{align*}
\|u\|_{L^\alpha(C^+)}&\lesssim \|\phi u\|_{L^\alpha}\lesssim \|\phi u\|_{L^\rho}^a\|\nabla (\phi u)\|_{L^\rho}^{1-a} \lesssim \|\phi u\|_{L^2}^{ab}\|\nabla (\phi u)\|_{L^2}^{a(1-b)}\|u\|_{W^{1,\rho}}^{1-a} \\&\lesssim \|\phi u\|_{L^2}^{ab} M^{a(1-b)}B(t)^{1-a} \lesssim \left(\frac{TM}{D} + \|u_0\|_{L^2(\real^d\setminus A^-)}\right)^{ab}M^{a(1-b)}B(t)^{1-a} \\ &\lesssim \delta(D)B(t)^{1-a}
\end{align*}
For the second term, taking $\epsilon>0$ small, for $|y|<1$,
\begin{align*}
\|u(\cdot-y)-u\|_{L^\rho(B^+)} &= \|u(\cdot-y)-u\|_{L^\rho(B^+)}^\epsilon \|u(\cdot-y)-u\|_{L^\rho(B^+)}^{1-\epsilon} \\&\lesssim \left(\|u\|_{L^\rho(C^+)}\right)^\epsilon\left(\|u(\cdot-y)-u\|_{L^\rho}\right)^{1-\epsilon}\\&\lesssim \left(\|\phi u\|_{L^2}^b\|\nabla (\phi u)\|_{L^2}^{1-b}\right)^\epsilon \left(\|\nabla u\|_{L^\rho}|y|\right)^{1-\epsilon}\\&\lesssim \delta(D) B(t)^{1-\epsilon} |y|^{1-\epsilon}
\end{align*}
Hence
\begin{align*}
|I_1|&\lesssim \delta(D) B(t)^{\sigma-a} \|v\|_{B^s_{\rho,2}} + \delta(D) B(t)^{\sigma+1-\epsilon}\left(\int_0^1 \frac{1}{\tau^{1+2s-2(1-\epsilon)}}d\tau\right)^{1/2} \\&\lesssim \delta(D)\left(B(t)^{\sigma+1-a}  + B(t)^{\sigma+1-\epsilon} \right)\\&\lesssim \delta(D)(1 + B(t)^{\sigma+1})
\end{align*}
and so
$$
\|I_1\|_{L^{\gamma'}(0,T)}\lesssim \delta(D)(T^{\frac{1}{\gamma'}} + T^{\frac{4-(d-2s)\sigma}{4}}) \lesssim \delta(D).
$$
For the $I_3$ term,
\begin{align*}
|I_3|&\lesssim \||u|^\sigma v\|_{L^{\rho'}(B^+)} + \||v|^\sigma u\|_{L^{\rho'}(B^+)} \lesssim \|u\|_{L^\alpha(B^+)}^\sigma\|v\|_\rho + \|v\|_\alpha^\sigma \|u\|_{L^\rho(B^+)}\\&\lesssim \left(\delta(D)B(t)^{1-a}\right)^\sigma\|v\|_{W^{1,\rho}} + \|v\|_{W^{1,\rho}}^\sigma\delta(D)\\&\lesssim \delta(D) B(t)^{\sigma(1-a)+1} + \delta(D)B(t)^\sigma \lesssim \delta(D)(1+B(t)^{\sigma+1}).
\end{align*}
As for the $I_1$ term, this implies $\|I_3\|_{L^{\gamma'}(0,T)}\lesssim \delta(D)$.
The estimates for $I_2$ and $I_4$ are analogous. Thus
$$
\|N(u,v)\|_{L^{\gamma'}((0,T),\dot{B}^s_{\rho',2})}\lesssim \delta(D).
$$
Finally,
\begin{align*}
\|N(u,v)\|_{L^{\gamma'}((0,T),B^s_{\rho',2})}&\lesssim \|N(u,v)\|_{L^{\gamma'}((0,T),L^{\rho'})} +  \|N(u,v)\|_{L^{\gamma'}((0,T),\dot{B}^s_{\rho',2})}\\&\lesssim \|I_3\|_{L^{\gamma'}(0,T)} + \|I_4\|_{L^{\gamma'}(0,T)} + \delta(D)\lesssim \delta(D).
\end{align*}
\end{proof}
\begin{proof}[Proof of Theorem \ref{teocritico} for $\sigma>4/d$]
We follow closely the proof of the $L^2$-critical case $\sigma=4/d$. Set $y=2De_1$, $u=\NLS(u_0(\cdot+y/2))$ and $v=\NLS(v_0(\cdot-y/2))$. Once again, consider the initial value problem
$$
iw_t + \Delta w + |u+v+w|^{\sigma}(u+v+w)-|u|^\sigma u - |v|^\sigma v=0, w(0)=w_0\in H^1.
$$
Define
$$
M(t)=\|u\|_{L^\gamma((0,t), W^{1,\rho})} + \|v\|_{L^\gamma((0,t), W^{1,\rho})}.
$$
Applying Strichartz estimates on the Duhamel formula for $w$ on a time interval $0<t<T$,
\begin{align*}
\|w\|_{S^s(0,t)} &\lesssim \|w_0\|_{H^s} + \left(\|u\|_{L^{\gamma}((0,t), B^s_{\rho,2})}^\sigma + \|v\|_{L^{\gamma}((0,t), B^s_{\rho,2})}^\sigma + \|w\|_{S^s(0,t)}^\sigma \right)\|w\|_{S^s(0,t)} \\&+ \|N(u,v)\|_{L^{\gamma'}((0,t), B^s_{\rho',2})}.
\end{align*}
where $N(u,v)$ is defined as in \eqref{eq:termofonte}. It follows from Lemmata \ref{lema:injeccao} and \ref{lema:estimativafonte} that
$$
\|w\|_{S^s(0,t)} \lesssim \|w_0\|_{H^s} +\delta(D)+ M(t)^\sigma\|w\|_{S^s(0,t)} + \|w\|_{S^s(0,t)}^{\sigma+1}.
$$
Choose $T_0$ such that
$$
M^\sigma(T_0)\lesssim \frac{1}{2}.
$$
Then
$$
\|w\|_{S^s(0,t)} \lesssim \|w_0\|_{H^s} +\delta(D) + \|w\|_{S^s(0,t)}^{\sigma+1}.
$$
An obstruction argument now implies that, if 
\begin{equation}\label{eq:condicaoobstrucao}
\|w_0\|_{H^s} +\delta(D)<\eta',\quad \eta'<\eta,
\end{equation}
then $w$ exists (as an $H^s$ solution) up to time $T_0$ and $\|w(T_0)\|_{H^s}<2\eta'$. For small enough $\|w_0\|_{H^s}$ and $D$ large, the process can be iterated so that $w$ is defined on $[0,T]$ and $\|w\|_{S^s(0,T)}<\epsilon$. The proof of the global existence is completely analogous to the proof for the critical case.
\end{proof}

\section{Further comments}

Consider the weakly coupled nonlinear Schrödinger system
\begin{equation}\tag{2-NLS}\label{2NLS}
\left\{\begin{array}{l}
iu_t + \Delta u + k_{11}|u|^{2p}u + k_{12}|v|^{p+1}|u|^{p-1}u = 0\\
iv_t + \Delta v + k_{22}|v|^{2p}v + k_{12}|u|^{p+1}|v|^{p-1}v = 0
\end{array}\right.,\ u,v\in C([0,T), H^1(\real^d)),
\end{equation}
where $k_{ij}\in \real$ and $1\le p<2/(d-2)^+$. Using the standard techniques available for the (NLS), one may show that the initial value problem is locally well-posed for $u_0, v_0\in H^s(\real^d)$ if $p<2/(d-2s)^+$ and is conditionally locally well-posed if $p=2/(d-2s)^+$. We set $T(u_0,v_0)$ as the maximal time of existence of the solution with initial conditions $u_0,v_0$ and write $(u,v)=(2\NLS)(u_0,v_0)$. Finally, write
$$
\mathcal{GD}_2=\left\{ (u_0,v_0)\in (H^1(\real^d))^2: T(u_0,v_0)=\infty,\ \|(2\NLS)(u_0,v_0)\|_{(S^1(0,\infty))^2}<\infty \right\}.
$$

It is easy to check that \eqref{2NLS} has finite speed of disturbance for each component: if $u_0$ has compact support and $B$ is a set such that $\dist(\supp u_0, B)>0$, then
$$
\|u(t)\|_{L^2(B)}\le \frac{2\sup_{s\in[0,t]}\|\nabla u(s)\|_2}{\dist(\supp u_0,B)}t.
$$
The same is valid for $v$. Consequently, one may prove the analogous concatenation result:
\begin{prop}
Set $p = 2/d$ or $p\ge 2$. Given two initial data $\mathbf{u}_0, \mathbf{v}_0\in (H^1(\real^d))^2$,  a fixed time $T<T(\mathbf{u}_0), T(\mathbf{v}_0)$ and $\epsilon>0$, there exists $D_T>0$ such that, for $\mathbf{w}_0\in (H^1(\real^d))^2$ sufficiently small 
$$
T(\mathbf{u}_0+\mathbf{v}_0(\cdot-y) + \mathbf{w}_0)>T, \quad |y|>D_T
$$
and, taking $s$ such that $p\ge 2/(d-2s)$,
$$
\| (2\NLS)(\mathbf{u}_0+\mathbf{v}_0(\cdot-y) + \mathbf{w}_0) - (2\NLS)(\mathbf{u}_0)-(2\NLS)(\mathbf{v}_0(\cdot-y))\|_{(S^s(0,T))^2} <\epsilon.
$$
Moreover, if $\mathbf{u}_0, \mathbf{v}_0\in \mathcal{GD}$, there exists $D_\infty>0$ such that $\mathbf{u}_0+\mathbf{v}_0(\cdot - y)\in \mathcal{GD}$, $|y|>D_\infty$, and 
$$
\| (2\NLS)(\mathbf{u}_0+\mathbf{v}_0(\cdot-y) + \mathbf{w}_0) - (2\NLS)(\mathbf{u}_0)-(2\NLS)(\mathbf{v}_0(\cdot-y))\|_{(S^s(0,\infty))^2} <\epsilon.
$$
\end{prop}
When $k_{11}, k_{22}<0$, any initial data of the form $\mathbf{u}_0=(u_0,0)$ or $\mathbf{v}_0=(0,v_0)$, with $u_0, v_0\in H^1(\real^d)\cap L^2(|x|^2dx)$, is in $\mathcal{GD}_2$: the system \eqref{2NLS} is reduced to a defocusing (NLS) and the solutions are global and present linear decay. As a consequence, we have
\begin{cor}
Set $p = 2/d$ or $p\ge 2$. Moreover, suppose that $k_{11}, k_{22}<0$. Given $u_0,v_0\in H^1(\real^d)\cap L^2(|x|^2dx)$, there exists $D_\infty$ such that $(u_0,v_0(\cdot-y))\in\mathcal{GD}_2$, for any $|y|>D_\infty$.
\end{cor} 
Thus blow-up behaviour can only appear if the initial supports of the two components are sufficiently close to each other. We recall that, if $k_{12}>0$ is large, blow-up behavior is possible, by the usual Virial argument.

\section{Acknowledgements}
The author was partially suported by Fundação para a Ciência e Tecnologia, through the grants UID/MAT/04561/2013 and SFRH/BD/96399/2013. The author is indebted to the referee for its numerous comments and suggestions.

\small
\noindent \textsc{Sim\~ao Correia}\\
CMAF-CIO and FCUL \\
\noindent Campo Grande, Edif\'icio C6, Piso 2, 1749-016 Lisboa (Portugal)\\
\verb"sfcorreia@fc.ul.pt"\\

\end{document}